\documentclass[letter, 12pt,round]{article}

\usepackage{natbib}
\usepackage{amssymb,amsmath,xcolor,graphicx,xspace,colortbl,rotating} %
\usepackage[raggedrightboxes]{ragged2e} 
\usepackage{textcomp}
\usepackage{appendix}  
\usepackage{bm}  
\usepackage{boxedminipage}  
\usepackage{color}  
\usepackage{endnotes}  

\usepackage{ragged2e}  
\usepackage[onehalfspacing]{setspace}  
\usepackage{tabulary}  
\usepackage{varioref}  
\usepackage{wrapfig}  

\graphicspath{{Stochastic_Path 17 new_graphics/}{Stochastic_Path 17 new_tcache/}{Stochastic_Path 17 new_gcache/}}
\DeclareGraphicsExtensions{.pdf,.eps,.ps,.png,.jpg,.jpeg}
\usepackage[paper=letterpaper, twoside=false,textwidth=6.3in, textheight=8.7in,left=0.98in, top=0.99in,headheight=0.17in, headsep=0.17in,]{geometry}
\usepackage {bm}
\usepackage {tabulary}

\newtheorem {theorem}{Theorem}

\newtheorem {assumption}{Assumption}

\newtheorem {corollary}{Corollary}

\newtheorem {definition}{Definition}

\newtheorem {proposition}{Proposition}

\newenvironment {proof}[1][Proof]{\noindent \textbf {#1.} }{\ \rule {0.5em}{0.5em}}
\begin{document}
\title{Stochastic Comparative Statics in Markov Decision Processes}

\author{Bar Light\protect\footnote{Graduate School of Business, Stanford University, Stanford, CA 94305, USA. e-mail: \textsf{barl@stanford.edu}\ }  }
\maketitle

\thispagestyle {empty}

\noindent \noindent \textsc{Abstract}: 

\begin{quote}
In multi-period stochastic optimization problems, the future optimal decision is a random variable whose distribution depends on the parameters of the optimization problem. We analyze how the expected value of this random variable changes as a function of the dynamic optimization parameters in the context of Markov decision processes. We call this analysis \emph{stochastic comparative statics}. We derive both \emph{comparative statics} results and \emph{stochastic comparative statics} results showing how the current and future optimal decisions change in response to changes in the single-period payoff function, the discount factor, the initial state of the system, and the transition probability function. We apply our results to various models from the economics and operations research literature, including investment theory, dynamic pricing models, controlled random walks, and comparisons of stationary distributions.  
\end{quote}

\smallskip \noindent MSC2000 subject classification:  90C40 \\
\smallskip \noindent OR/MS subject classification:  Primary: Dynamic programming/optimal control

\newpage 

\section{Introduction}
A question of interest in a wide range of problems in economics and operations research is whether the solution to an optimization problem is monotone with respect to its parameters. The analysis of this question is called \emph{comparative statics}.\protect\footnote{
See \cite{topkis2011supermodularity} for a comprehensive treatment of comparative statics methods.
} Following Topkis' seminal work \citep{topkis1978minimizing},  comparative statics methods have received significant attention in the economics and operations research literature.\protect\footnote{
See for example \cite{licalzi1992subextremal}, \cite{milgrom1994monotone}, \cite{athey2002monotone}, \cite{echenique2002comparative}, \cite{antoniadou2007comparative}, \cite{quah2007comparative}, \cite{quah2009comparative}, \cite{shirai2013welfare}, \cite{nocetti2015robust}, \cite{wang2015precautionary}, \cite{barthel2018directional}, and \cite{koch2019index}. } While comparative statics methods are usually applied to static optimization problems, they can also be applied to dynamic optimization problems. In particular, these methods can be used to study how the policy function\footnote{\cite{muller1997does} and \cite{smith2002structural} study how the optimal value function changes with respect to the parameters of the dynamic optimization problem, such as the single-period payoff function and the transition probability function. In contrast, in this paper, we analyze the optimal policy function.} changes with respect to the current state of the system or with respect to other parameters of the dynamic optimization problem.\protect\footnote{
For comparative statics results in dynamic optimization models see \cite{serfozo1976monotone}, \cite{lovejoy1987ordered}, \cite{amir1991one}, \cite{hopenhayn1992stochastic},  \cite{mirman2008qualitative},  \cite{topkis2011supermodularity}, \cite{krishnamurthy2016partially},
\cite{smith2017risk},
\cite{lehrer2018effect}, and
\cite{l2017supermodular}.} That is, for multi-period optimization models, comparative statics methods can be used to determine how the current period's optimal decision changes with respect to the parameters of the optimization problem. For example, in a Markov decision process, under suitable conditions on the payoff function and on the transition function, comparative statics methods can be applied to show that the optimal decision is increasing in the discount factor when the state of the system is fixed. But since the model is dynamic and includes uncertainty, the states' evolution is different under different discount factors, and thus, it is not clear whether the future optimal decision is increasing in the discount factor even when the current optimal decision is increasing in the discount factor for a fixed state. 

The state of the system in period $t >1$ is a random variable from the point of view of period $1$, and thus, the optimal decision in period $t$, which depends on the state of the system in period $t$, is a random variable given the information available in period $1$. In this paper, we analyze how the expected value of the optimal decision in period $t$ changes as a function of the optimization problem parameters in the context of Markov decision processes (MDP). We call this analysis \emph{stochastic comparative statics.} More precisely, let $(E,\succeq )$ be a partially ordered set that contains some parameters of the MDP. For example, $E$ can be the set of all transition probability functions, the set of all discount factors, and/or a set of parameters that influence the payoff function. Suppose that under the parameters $e \in E$ a stationary policy function is given by $g(s,e)$ where $s$ is the state of the system. Given the policy function $g$ and the system's initial state, the system's states follow a stochastic process. Suppose that the states' distribution in period $t$ is  described by the probability measure $\mu ^{t}(ds,e)$. We are interested in finding conditions that ensure that the expected decision in period $t$, $\mathbb{E}_{\,}^{t}(g(e)) =\int g(s ,e)\mu _{\,}^{t}(ds ,e)$ is increasing in the parameters $e$ on $E$. 

The expected value $\mathbb{E}_{\,}^{t}(g(e))$ is interpreted in two different ways. From a probabilistic point of view, $\mathbb{E}_{\,}^{t}(g(e))$ is the expected optimal decision in period $t$ as a function of the parameters $e$. For example, in investment theory, this expected value usually represents the expected capital accumulation in the system in period $t$  \citep{stokey1989}. In inventory management, it represents the expected inventory in period $t$ \citep{krishnan2010inventory}, and in income fluctuation problems it represents the expected wealth accumulation (see \cite{huggett2004precautionary} and \cite{bommier2018risk}) in period $t$. From a deterministic point of view, if we consider a population of ex-ante identical agents whose states evolve independently according to the stochastic process that governs the states' dynamics, then $\mu ^{t}$ represents the empirical distribution of states in period $t$. In this case, $\mathbb{E}_{\,}^{t}(g(e))$ corresponds to the average decision in period $t$ of this population given the parameters $e$. This latter interpretation is common in the growing literature on stationary equilibrium models and mean field equilibrium models. In this literature, while the focus is on the analysis of equilibrium, some stochastic comparative statics results have been obtained (see \cite{adlakha2013mean} and \cite{acemoglu2015robust}). These stochastic comparative statics results are useful in analyzing the equilibrium of these models. In particular, proving comparative statics results and establishing the uniqueness of an equilibrium (see \cite{hopenhayn1992entry}, \cite{light2017uniqueness},  \cite{acemoglu2018equilibrium}, and \cite{light2018mean}).

The goal of this paper is to provide general stochastic comparative statics results in the context of an MDP. In particular, we provide various sufficient conditions on the primitives of MDPs that guarantee stochastic comparative statics results with respect to important parameters of MDPs, such as the discount factor, the single-period payoff function, and the transition probability function. We also provide novel comparative statics results with respect to these parameters. For example, we show that under a standard set of conditions that implies that the policy function is increasing in the state, the policy function is increasing the discount factor also (see Section \ref{Section: discount}). We apply our results in capital accumulation models with adjustment costs \citep{hopenhayn1992stochastic}, in dynamic pricing models with reference effects \citep{popescu2007dynamic}, and in controlled random walks. As an example, consider the following controlled random walk $s_{t+1} =s_{t} +a_{t} +\epsilon _{t+1}$ where $s_{t}$ is the state of the system in period $t$, $a_{t}$ is the action chosen in period $t$, and $\{\epsilon _{t}\}_{t =1}^{\infty}$ are random variables that are independent and identically distributed across time. In each period, a decision maker receives a reward that depends on the current state of the system and incurs a cost that depends on the action that the decision maker chooses in that period. The reward function is increasing in the state of the system and the cost function is increasing in the decision maker's action. The decision maker's goal is to maximize the expected sum of rewards. We provide sufficient conditions on the reward function and on the cost function that guarantee that the decision maker's current action and the expected future actions increase when the distribution of the random noise $\epsilon$ is higher in the sense of stochastic dominance.
Since our results are intuitive and the sufficient conditions that we provide in order to derive stochastic comparative statics results are satisfied in some dynamic programs of interest, we believe that our results hold in other applications as well.

The rest of the paper is organized as follows. Section \ref{Section MODEL} presents the dynamic optimization model. Section \ref{Section notations} presents definitions and notations that are used throughout the paper. In Section \ref{SCS} we present our main stochastic comparative statics results. In Section \ref{Section: discount} we study changes in the discount factor and in the single-period payoff function. In Section \ref{Section: transition} we study changes in the transition probability function. In Section \ref{Sectopn applications} we apply our results to various models. In Section \ref{Section final} we provide a summary, followed by an Appendix containing proofs.

\section{\label{Section MODEL}The model}
In this section we present the main components and assumptions of the model.  For concreteness, we focus on a standard discounted dynamic programming model, sometimes called a Markov decision process.\protect\footnote{
All our results can be applied to other dynamic programming models, such as positive dynamic programming and negative dynamic programming.
} For a comprehensive treatment of dynamic programming models, see \cite{feinberg2012handbook} and \cite{puterman2014markov}.

We define a discounted dynamic programming model in terms of a tuple of elements $(S ,A ,\Gamma  ,p ,r ,\beta )$. $S \subseteq \mathbb{R}^{n}$ is a Borel set called the state space. $\mathcal{B}(S)$ is the Borel $\sigma $-algebra on $S$. $A \subseteq \mathbb{R}$ is the action space.
$\Gamma$ is a measurable subset of $S \times A$. For all $s \in S$, the non-empty and measurable $s$-section $\Gamma(s)$ of $\Gamma$ is the set of  feasible actions in state $s \in S$. $p :S \times A \times \mathcal{B}(S) \rightarrow [0 ,1]$ is a transition probability function. That is, $p(s ,a , \cdot)$ is a probability measure on $S$ for each $(s ,a) \in S \times A$ and $p(\cdot, \cdot,B)$ is a  measurable function for each $B \in \mathcal{B}(S)$. $r :S \times A \rightarrow \mathbb{R}$ is a measurable single-period payoff function. $0 <\beta  <1$ is the discount factor.  


There is an infinite number of periods $t \in \mathbb{N} : =\{1 ,2 , . . .\}$. The process starts at some state $s(1) \in S$. Suppose that at time $t$ the state is $s(t)$. Based on $s(t)$, the decision maker (DM) chooses an action $a(t) \in \Gamma (s(t))$ and receives a payoff $r(s(t) ,a(t))$. The probability that the next period's state $s(t +1)$ will lie in $B \in \mathcal{B}(S)$ is given by $p(s(t) ,a(t) ,B)$.

Let $H =S \times A$ and $H^{t}:=\underbrace{H \times \ldots  \times H}_{t -1~ \mathrm{t} \mathrm{i} \mathrm{m} \mathrm{e} \mathrm{s}} \times S$. A policy $\sigma$ is a sequence $(\sigma _{1} ,\sigma _{2} , \ldots )$ of Borel measurable functions $\sigma _{t} :H^{t} \rightarrow A$ such that $\sigma _{t}(s(1) ,a(1) ,\ldots  ,s(t)) \in \Gamma (s(t))$ for all $t \in \mathbb{N}$ and all $(s(1) ,a(1) ,\ldots  ,s(t))\in H^{t}$. For each initial state $s\left (1\right )$, a policy $\sigma $ and a transition probability function $p$ induce a probability measure over the space of all infinite histories $H^{\infty }$.\protect\footnote{
The probability measure on the space of all infinite histories $H^{\infty }$ is uniquely defined by the Ionescu Tulcea theorem (for more details, see \cite{bertsekas1978stochastic} and \cite{feinberg1996measurability}).
} We denote the expectation with respect to that probability measure by $\mathbb{E}_{\sigma }$, and the associated stochastic process by $\{s(t) ,a(t)\}_{t =1}^{\infty }$. The DM's goal is to find a policy that maximizes his expected discounted payoff. When the DM follows a strategy $\sigma$ and the initial state is $s \in S$ his expected discounted payoff is given by
\begin{equation*}V_{\sigma}(s) =\mathbb{E}_{\sigma }\sum \limits_{t =1}^{\infty }\beta^{t -1}r(s(t),a(t)).
\end{equation*}Define \begin{equation*}V(s) =\sup _{\sigma }V_{\sigma }(s).
\end{equation*}
We call $V :S \rightarrow \mathbb{R}$ the value function. 

Define the operator $T:B(S)\rightarrow B(S)$ where $B(S)$ is the space of all functions $f :S \rightarrow \mathbb{R}$ by 
\begin{equation*}Tf(s) =\max _{a \in \Gamma (s)}h(s ,a ,f), 
\end{equation*}
where 
\begin{equation}h(s ,a ,f) =r(s ,a) +\beta \int _{S}f(s^{ \prime })p(s ,a ,ds^{ \prime }). \label{eq:h}
\end{equation}
Under standard assumptions on the primitives of the MDP,\footnote{The state and action spaces can be continuous or discrete. When we discuss  convex functions on $S$ we assume that $S$ is a convex set.} standard dynamic programming arguments show that the value function $V$ is the unique  function that satisfies  $TV =V$. In addition, there exists an optimal stationary policy and the optimal policies correspondence 
\begin{equation*}G(s) =\{a \in \Gamma (s) :V(s) =h(s ,a ,V)\} 
\end{equation*}
is nonempty, compact-valued and upper hemicontinuous. Define $g(s) =\max G(s)$. We call $g(s)$ the policy function. For the rest of the paper, we assume that the value function is the unique and continuous function that satisfies $TV =V$, $T^{n}f$ converges uniformly to $V$ for every $f \in B(S)$, and that the policy function exists.\footnote{These conditions are usually satisfied in applications. Conditions that ensure the existence and continuity of the value function and the existence of a stationary policy function are widely studied in the literature. See \cite{hinderer2016dynamic} for a textbook treatment. For recent results, see \cite{feinberg2016partially} and references therein.}

\subsection{\label{Section notations} Notations and definitions }
In this paper we consider a parameterized dynamic program. Let $(E , \succeq )$ be a partially ordered set that influences the DM's decisions. We denote a generic element in $E$ by $e$. Throughout the paper, we slightly abuse the notations and allow an additional argument in the functions defined above. For instance, the value function of the parameterized dynamic program $V$ is denoted by
\begin{equation*}V (s ,e) =\max_{a \in \Gamma  (s ,e)}h (s ,a ,e ,V).
\end{equation*} 
Likewise, the policy function is denoted by $g(s ,e)$; $r(s ,a ,e)$ is the single-period payoff function; and $h (s ,a ,e ,V)$ is the $h$ function associated with the dynamic program problem with parameters $e$, as defined above in Equation (\ref{eq:h}). For the rest of the paper, we let $E_{p}$ be the set of all transition functions $p:S \times A \times \mathcal{B}(S) \rightarrow [0 ,1]$.

When the DM follows the policy function $g(s)$ and the initial state is $s(1)$, the stochastic process $(s(t))$ is a Markov process. The transition function of $(s(t))$ can be described by the policy function $g$ and by the transition function $p$ as follows: For all $B \in \mathcal{B}(S)$, define $\mu ^{1}(B) =1$ if $s(1) \in B$ and $0$ otherwise, and $\mu ^{2}(B) =p(s(1) ,g(s(1)) ,B)$. $\mu ^{2}(B)$ is the probability that the second period's state $s(2)$ will lie in $B$. For $t \geq 3$, define $\mu ^{t}(B) =\int _{S}p(s ,g(s) ,B)\mu ^{t -1}(ds)$ for all $B \in \mathcal{B}(S)$. Then $\mu ^{t}(B)$ is the probability that $s(t)$ will lie in $B \in \mathcal{B}(S)$ in period $t$ when the initial state is $s(1) \in S$ and the DM follows the policy function $g$. For notational convenience, we omit the reference to the initial state. All the results in this paper hold for every initial state $s(1)\in S$. 

We write $\mu _{i}^{t}(B)$ to denote the probability that $s$ will lie in $B \in \mathcal{B}(S)$ in period $t$, when $e_{i} \in E$ are the parameters that influence the DM's decisions and the DM follows the policy function $g(s ,e_{i})$, $i =1 ,2$. For $e_{i} \in E$, define  
\begin{equation*}\mathbb{E}_{i}^{t}(g(e_{i})) =\int _{S}g(s ,e_{i})\mu _{i}^{t}(ds) .
\end{equation*}As we discussed in the introduction, $\mathbb{E}_{i}^{t}(g(e_{i}))$ can be interpreted in two ways. According to the first interpretation, the DM's optimal decision in period $t$ is a random variable from the point of view of period $1$. The expected value $\mathbb{E}_{i}^{t}(g(e_{i}))$ is the DM's expected decision in period $t$, given that the parameters that influence the DM's decisions are $e_{i} \in E$. Alternately, the expected value $\mathbb{E}_{i}^{t}(g(e_{i}))$ can be interpreted as the aggregate of the decisions of a continuum of DMs facing idiosyncratic shocks. In the latter interpretation, each DM has an individual state and $\mu ^{t}$ is the distribution of the DMs over the states in period $t$. This interpretation is often used in stationary equilibrium models and in mean field equilibrium models (see more details in Section \ref{Section 3.3}). We are interested in the following stochastic comparative statics question: is it true that $e_{2} \succeq e_{1}$ implies $\mathbb{E}_{2}^{t}(g(e_{2})) \geq \mathbb{E}_{1}^{t}(g(e_{1}))$ for all $t \in \mathbb{N}$ (and for each initial state)? We note that for $t =1$, the stochastic comparative statics question reduces to a comparative statics question: is it true that $e_{2} \succeq e_{1}$ implies $g(s ,e_{2}) \geq g(s ,e_{1})$?

We now introduce some notations and definitions that will be used in the next sections. 

For two elements $x ,y \in \mathbb{R}^{n}$ we write $x \geq y$ if $x_{i} \geq y_{i}$ for each $i =1 , . . . ,n$. We say that $f :\mathbb{R}^{n} \rightarrow \mathbb{R}$ is increasing if $x \geq y$ implies $f(x) \geq f(y)$.

Let  $D \subseteq \mathbb{R}^{S}$ where $\mathbb{R}^{S}$ is the set of all functions from $S$ to $\mathbb{R}$. When $\mu _{1}$ and $\mu _{2}$ are probability measures on $(S ,\mathcal{B}(S))$, we write $\mu _{2} \succeq _{D}\mu _{1}$ if \begin{equation*}\int _{S}f(s)\mu _{2}(ds) \geq \int _{S}f(s)\mu _{1}(ds)
\end{equation*}for all Borel measurable functions $f \in D$ such that the integrals exist.

In this paper we will focus on two important stochastic orders: the first order stochastic dominance and the convex stochastic order. When $D$ is the set of all increasing functions on $S$, we write $\mu _{2} \succeq _{st}\mu _{1}$ and say that $\mu _{2}$ first order stochastically dominates $\mu _{1}$. If $D$ is the set of all convex functions on $S$, we write $\mu _{2} \succeq _{CX}\mu _{1}$ and say that $\mu _{2}$ dominates $\mu _{1}$ in the convex stochastic order. If $D$ is the set of all increasing and convex functions on $S$, we write $\mu _{2} \succeq _{ICX}\mu _{1}$. Similarly, for $p_{1} ,p_{2} \in E_{p}$, we write $p_{2} \succeq _{D}p_{1}$ if \begin{equation*}\int _{S}f(s^{ \prime })p_{2}(s,a ,ds^{ \prime }) \geq \int _{S}f(s^{ \prime })p_{1}(s ,a ,ds^{ \prime })
\end{equation*}for all Borel measurable functions $f \in D \subseteq \mathbb{R}^{S}$ and all $(s ,a) \in S \times A$ such that the integrals exist.\protect\footnote{
In the rest of the paper, all functions are assumed to be integrable.  
} If $D$ is the set of all increasing functions, convex functions, and convex and increasing functions, we write $p_{2} \succeq _{st}p_{1}$, $p_{2} \succeq _{CX}p_{1}$, and $p_{2} \succeq _{ICX}p_{1}$, respectively. For  comprehensive coverage of stochastic orders and their applications, see \cite{muller2002comparison} and \cite{shaked2007stochastic}.

\begin{definition}
(i) We say that $p \in E_{p}$ is monotone if for every increasing function $f$ the function $\int _{S}f(s^{ \prime })p(s ,a ,ds^{ \prime })$ is increasing in $(s ,a)$.

(ii) We say that $p \in E_{p}$ is convexity-preserving if for every convex function $f$ the function $\int _{S}f(s^{ \prime })p(s ,a ,ds^{ \prime })$ is convex in $(s ,a)$. 

(iii) Define $P_{i}(s ,B) = :p_{i}(s ,g(s ,e_{i}) ,B)$. Let $D \subseteq \mathbb{R}^{S}$. We say that $P_{i}$ is $D$-preserving if $f \in D$ implies that $\int _{S}f(s^{ \prime })P_{i}(s ,ds^{ \prime }) \in D$.  If $D$ is the set of all increasing functions, convex functions, and convex and increasing functions, we say that $P_{i}$ is $I$-preserving, $CX$-preserving, and $ICX$-preserving, respectively.
\end{definition}

\section{\label{Section dynamics}Main results}

In this section we derive our main results. In Section 3.1 we provide stochastic comparative statics results. In Section 3.2 and in Section 3.3 we provide conditions on the primitives of the MDP that guarantee comparative statics and stochastic comparative statics results. 

\subsection{\label{SCS}Stochastic comparative statics}

 In this section we provide conditions that ensure stochastic comparative statics. Our approach is to find conditions that imply that the states' dynamics generated under $e_{2}$ stochastically dominate the states' dynamics generated under $e_{1}$ whenever $e_{2} \succeq e_{1}$. Theorem \ref{Theorem 1} shows that if $P_{2}$ is $D$-preserving and $P_{2}(s , \cdot ) \succeq _{D}P_{1}(s , \cdot )$ for all $s \in S$, then $\mu _{2}^{t} \succeq _{D}\mu _{1}^{t}$ for all $t \in \mathbb{N}$. A proof of Theorem \ref{Theorem 1} can be found in Chapter 5 in \cite{muller2002comparison} where the authors study stochastic comparisons of general Markov chains. For completeness, because our setting is slightly different, we provide the proof of Theorem \ref{Theorem 1} in the Appendix for completeness.\footnote{A similar result to Theorem  \ref{Theorem 1} for the case of $ \succeq _{st}$ and $ \succeq _{ICX}$ can be found in \cite{huggett2004precautionary}, \cite{adlakha2013mean}, \cite{balbus2014constructive}, and \cite{acemoglu2015robust}.}  

The focus of the rest of the paper is on finding sufficient conditions on the primitives of the MDP in order to apply Theorem \ref{Theorem 1}. Corollary \ref{Parameter} and Theorem \ref{TRANSITION} provide sufficient conditions for $P_{2}$ to be $D$-preserving and $P_{2}(s , \cdot ) \succeq _{D}P_{1}(s , \cdot )$ when $D$ is the set of increasing functions or the set of increasing and convex functions. The results in this section require conditions on the policy function and on the primitives of the MDP. In Sections \ref{Section: discount} and \ref{Section: transition}, we provide comparative statics and stochastic comparative statics results that depend only on the primitives of the model (e.g., the transition probabilities and the single-period payoff function).

\begin{theorem}
\label{Theorem 1}Let $(E , \succeq )$ be a partially ordered set and let $D \subseteq \mathbb{R}^{S}$. Let $e_{1} ,e_{2}\in E$ and suppose that $e_{2} \succeq e_{1}$. Assume that $P_{2}$ is $D$-preserving and that $P_{2}(s , \cdot ) \succeq _{D}P_{1}(s , \cdot )$ for all $s \in S$. Then $\mu _{2}^{t} \succeq _{D}\mu _{1}^{t}$ for all $t \in \mathbb{N}$.  
\end{theorem}

In the case that $p_{2} =p_{1} =p$ and $(E , \succeq )$ is a partially ordered set that influences the agent's decisions, Theorem \ref{Theorem 1} yields a simple stochastic comparative statics result. Corollary \ref{Parameter} shows that if $g(s,e)$ is increasing in $e$, $g(s ,e_{2})$ is increasing in $s$, and $p$ is monotone, then $\mathbb{E}_{2}^{t}(g(e_{2})) \geq \mathbb{E}_{1}^{t}(g(e_{1}))$ whenever $e_{2} \succeq e_{1}$. This result is useful when $E$ is the set of all possible discount factors between $0$ and $1$, or is a set that includes parameters that influence the single-period payoff function (see Section \ref{Section: discount}).

\begin{corollary}
\label{Parameter}Let $e_{1} ,e_{2} \in E$ and suppose that $e_{2} \succeq e_{1}$. Assume that $g(s ,e)$ is increasing in $e$ for all $s \in S$, $g(s ,e_{2})$ is increasing in $s$,  $p_{1} =p_{2} =p$, and $p$ is monotone. Then 
\begin{equation*}\mathbb{E}_{2}^{t}(g(e_{2})) \geq \mathbb{E}_{1}^{t}(g(e_{1})) 
\end{equation*}for all $t \in \mathbb{N}$ and for each initial state $s(1) \in S$.
\end{corollary}

In some dynamic programs we are interested in knowing how a change in the initial state will influence the DM's decisions in future periods. Corollary \ref{Initial state} shows that a higher initial state  leads to higher expected decisions if the policy function is increasing in the state of the system and the transition probability function is monotone. The proof follows from the same arguments as those in the proof of Corollary \ref{Parameter}. Recall that we denote the initial state by $s(1)$.

\begin{corollary}
\label{Initial state} Consider two MDPs that are equivalent except for the initial states $s_{i}(1)$, $i =1 ,2$. Assume that $s_{2}(1) \geq s_{1}(1)$, $g(s)$ is increasing in $s$, and $p$ is monotone. Then  $\mathbb{E}_{2}^{t}(g(s_{2}(1))) \geq \mathbb{E}_{1}^{t}(g(s_{1}(1)))$ for all $t \in \mathbb{N}$.  
\end{corollary}

We now derive stochastic comparative statics results with respect to the transition probability function that governs the states' dynamics. Part (i) of Theorem \ref{TRANSITION} provides conditions that ensure that $p_{2} \succeq _{st}p_{1}$ implies $\mathbb{E}_{2}^{t}(g(p_{2})) \geq \mathbb{E}_{1}^{t}(g(p_{1}))$ for all $t \in \mathbb{N}$. Part (ii) provides conditions that ensure that $p_{2} \succeq _{CX}p_{1}$ implies $\mathbb{E}_{2}^{t}(g(p_{2})) \geq \mathbb{E}_{1}^{t}(g(p_{1}))$ for all $t \in \mathbb{N}$. In Section 4 we apply these results to various commonly studied dynamic optimization models.       

\begin{theorem}
\label{TRANSITION}Let $p_{1} ,p_{2} \in E_{p}$.

(i) Assume that $p_{2}$ is monotone, $g(s ,p_{2})$ is increasing in $s$, and $g(s ,p_{2}) \geq g(s ,p_{1})$ for all $s \in S$. Then $p_{2} \succeq _{st}p_{1}$ implies that $\mathbb{E}_{2}^{t}(g(p_{2})) \geq \mathbb{E}_{1}^{t}(g(p_{1}))$ for all $t \in \mathbb{N}$.

(ii) Assume that $p_{2}$ is monotone and convexity-preserving, $g(s ,p_{2})$ is increasing and convex in $s$, and $g(s ,p_{2}) \geq g(s ,p_{1})$ for all $s \in S$. Then $p_{2} \succeq _{CX}p_{1}$ implies that $\mathbb{E}_{2}^{t}(g(p_{2})) \geq \mathbb{E}_{1}^{t}(g(p_{1}))$ for all $t \in \mathbb{N}$.
\end{theorem}

\subsection{\label{Section: discount}A change in the discount factor or in the payoff function}
In this section we provide sufficient conditions for the monotonicity of the policy function in the state variable, and for the monotonicity of the policy function in other parameters of the MDP, including the discount factor and the parameters that influence the single-period payoff function. Our stochastic comparative statics results in Section \ref{SCS} rely on these monotonicity properties. Thus, we provide  conditions on the model's primitives that ensure stochastic comparative statics results.

The monotonicity of the policy function in the state variable follows from the conditions on the model's primitives provided in \cite{topkis2011supermodularity}. We note that these conditions are not necessary for deriving monotonicity results regarding the policy function, and in some specific applications one can still derive these monotonicity results using different techniques or under different assumptions.\protect\footnote{
For example, see \cite{lovejoy1987ordered} and \cite{hopenhayn1992stochastic}. See also \cite{smith2002structural} for conditions that guarantee that the value function is monotone and has increasing differences.} 

Recall that a function $f :S \times E \rightarrow \mathbb{R}$ is said to have increasing differences in $(s ,e)$ on $S \times E$ if for all $e_{2} ,e_{1} \in E$ and $s_{2} ,s_{1} \in S$ such that $e_{2} \succeq e_{1}$ and $s_{2} \geq s_{1}$, we have \begin{equation*}f(s_{2} ,e_{2}) -f(s_{2} ,e_{1}) \geq f(s_{1} ,e_{2}) -f(s_{1} ,e_{1}) .
\end{equation*} A function $f$ has decreasing differences if $-f$ has increasing differences.

A set $B \in \mathcal{B}(S)$ is called an upper set if $s_{1} \in B$ and $s_{2} \geq s_{1}$ imply $s_{2} \in B$. The transition probability $p \in E_{p}$ has stochastically increasing differences if $p(s ,a ,B)$ has increasing differences for every upper set $B$. See \cite{topkis2011supermodularity} for examples of transition probabilities that have stochastically increasing differences. The optimal policy correspondence $G$ is said to be ascending if $s_{2} \geq s_{1}$, $b \in G(s_{1})$, and $b^{ \prime } \in G(s_{2})$ imply $\max \{b ,b^{ \prime }\} \in G(s_{2})$ and $\min \{b ,b^{ \prime }\} \in G(s_{1})$. In particular, if $G$ is ascending, then $\min G(s)$ and $\max G(s)$ are increasing functions. \cite{topkis2011supermodularity} provides conditions under which the optimal policy correspondence $G$ is ascending. These conditions are summarized in the following assumption:

\begin{assumption}
\label{Ass Topkis}(i) $r(s ,a)$ is increasing in $s$ and has increasing differences.

(ii) $p$ is monotone and has stochastically increasing differences.

(iii) For all $s_{1} ,s_{2} \in S$,  $s_{1} \leq s_{2}$ implies $\Gamma (s_{1}) \subseteq \Gamma (s_{2})$.
\end{assumption}

Theorem \ref{Thorem DISCOUNT} shows that under Assumption \ref{Ass Topkis}, the policy function $g(s,\beta )$ is increasing in the discount factor. Furthermore, if the single period payoff function $r(s ,a ,c)$ depends on  some parameter $c$ and has increasing differences, then the policy function is increasing in the parameter $c$.

\begin{theorem}
\label{Thorem DISCOUNT}Suppose that Assumption \ref{Ass Topkis} holds and that $\Gamma(s)$ is ascending.  

(i) Let $0 <\beta _{1} \leq \beta _{2} <1$. Then $g(s ,\beta _{2}) \geq g(s ,\beta _{1})$ for all $s \in S$ and $\mathbb{E}_{2}^{t}(g(\beta _{2})) \geq \mathbb{E}_{1}^{t}(g(\beta _{1}))$ for all $t \in \mathbb{N}$.

(ii) Let $c \in E$ be a parameter that influences the payoff function. If the payoff function $r(s ,a ,c)$ has increasing differences in $(a ,c)$ and in $(s ,c)$, then $g(s ,c_{2}) \geq g(s ,c_{1})$ for all $s \in S$, and $\mathbb{E}_{2}^{t}(g(c_{2})) \geq \mathbb{E}_{1}^{t}(g(c_{1}))$ for all $t \in \mathbb{N}$ whenever $c_{2} \succeq c_{1}$.
\end{theorem}

\subsection{\label{Section: transition}A change in the transition probability function}
In this section we study stochastic comparative statics results related to a change in the transition function. We provide conditions on the transition function and on the payoff function that ensure that $p_{2} \succeq _{st}p_{1}$ implies comparative statics results and stochastic comparative statics results. We assume that the transition function $p_{i}$ is given by $p_{i}(s ,a ,B) =\Pr (m(s ,a ,\epsilon ) \in B)$ for all $B \in \mathcal{B}(S)$, where $\epsilon $ is a random variable with law $v$ and support $\mathcal{V} \subseteq \mathbb{R}^{k}$. Theorem \ref{Theorem Transition} provides conditions on the function $m$ that imply that the policy function is higher when $v$ is higher in the sense of stochastic dominance. In Section \ref{Sec: cont ran walk}, we provide an example of a controlled random walk where the conditions on $m$ are satisfied.

\begin{theorem}
\label{Theorem Transition}Suppose that $p_{i}(s,a ,B) =\Pr (m(s ,a ,\epsilon _{i}) \in B)$ where $m$ is convex, increasing, continuous, and has increasing differences in $(s ,a)$, $(s ,\epsilon )$ and $(a ,\epsilon )$; and $\epsilon _{i}$ has the law $v_{i}$, $i =1 ,2$. $r(s ,a)$ is convex and increasing in $s$ and has increasing differences.  For all $s_{1}, s_{2} \in S$, we have $\Gamma (s_{1}) = \Gamma (s_{2})$.

If $v_{2} \succeq _{st}v_{1}$ then

(i) $g(s ,p_{2}) \geq g(s ,p_{1})$ for all $s \in S$ and $g(s ,p_{2})$ is increasing in $s$.

(ii) $\mathbb{E}_{2}^{t}(g(p_{2})) \geq \mathbb{E}_{1}^{t}(g(p_{1}))$ for all $t \in \mathbb{N}$.
\end{theorem}

\section{\label{Sectopn applications} Applications}
In this section we apply our results to several dynamic optimization models from the economics and operations research literature.

\subsection{Capital accumulation with adjustment costs}
Capital accumulation models are widely studied in the investment theory literature \citep{stokey1989}. We consider a standard capital accumulation model with adjustment costs \citep{hopenhayn1992stochastic}. In this model, a firm maximizes its expected discounted profit over an infinite horizon. The single-period revenues depend on the demand and on the firm's capital. The demand evolves exogenously in a Markovian fashion. In each period, the firm decides on the next period's capital level and incurs an adjustment cost that depends on the current capital level and on the next period's capital level. Using the stochastic comparative statics results developed in the previous section, we find conditions that ensure that higher future demand, in the sense of first order stochastic dominance, increases the expected long run capital accumulated. We provide the details below.

Consider a firm that maximizes its expected discounted profit. The firm's single-period payoff function $r$ is given by \begin{equation*}r(s ,a) =R(s_{1} ,s_{2}) -c(s_{1} ,a)
\end{equation*} where $s=(s_{1},s_{2})$. The revenue function $R$ depends on an exogenous demand shock $s_{2} \in S_{2} \subseteq \mathbb{R}^{n -1}$, and on the current firm's capital stock $s_{1} \in S_{1} \subseteq \mathbb{R}_{ +}$. The state space is given by $S = S_{1} \times S_{2}$. The demand shocks follow a Markov process with a transition function $Q$. The firm chooses the next period's capital stock $a \in \Gamma (s_{1})$ and incurs an adjustment cost of $c(s_{1} ,a)$. The transition probability function $p$ is given by \begin{equation*}p(s ,a ,B) =1_{D}(a)Q(s_{2} ,C) ,
\end{equation*}where $D \times C =B$, $D$ is a measurable set in $\mathbb{R}$, $C$ is a measurable set in $\mathbb{R}^{n -1}$, and $Q$ is a Markov kernel on $S_{2} \subseteq \mathbb{R}^{n -1}$.  

It is easy to see that if $Q$ is monotone then $p(s ,a ,B) =1_{D}(a)Q(s_{2} ,C)$ is monotone and that $Q_{2} \succeq _{st}Q_{1}$ implies $p_{2} \succeq _{st}p_{1}$.

Assume that the revenue function $R$ is continuous and has increasing differences, that $c$ is continuous and has decreasing differences, and that $\Gamma (s)$ is ascending. Under these conditions, \cite{hopenhayn1992stochastic} show that the policy function $g(s ,p)$ is increasing in $s$ if $Q$ is monotone. If, in addition, $Q_{2} \succeq _{st}Q_{1}$, then $g(s ,p_{2}) \geq g(s ,p_{1})$ for all $s$ (see Corollary 7 in \cite{hopenhayn1992stochastic}). Thus, part (i) in Theorem \ref{TRANSITION} implies that $\mathbb{E}_{2}^{t}(g(p_{2})) \geq \mathbb{E}_{1}^{t}(g(p_{1}))$ for all $t \in \mathbb{N}$.

\begin{proposition}
Let $Q_{1}$ and $Q_{2}$ be two Markov kernels on $S_{2}$. Assume that $R$ is continuous and has increasing differences, $c$ is continuous and has decreasing differences, $\Gamma (s)$ is ascending, and $\Gamma (s_{1}) \supseteq \Gamma (s_{1}^{\prime})$ whenever $s_{1} \geq s_{1}^{\prime}$. Assume that $Q_{2}$ is monotone and that $Q_{2} \succeq _{st}Q_{1}$. Then under $Q_{2}$ the expected capital accumulation is higher than under $Q_{1}$, i.e., $\mathbb{E}_{2}^{t}(g(p_{2})) \geq \mathbb{E}_{1}^{t}(g(p_{1}))$ for all $t \in \mathbb{N}$.
\end{proposition}

\subsection{\label{Section Popescu} Dynamic pricing with a reference effect and an uncertain memory factor}
In this section we consider a dynamic pricing model with a reference effect as in \cite{popescu2007dynamic}. In this model the demand is sensitive to the firm's pricing history. In particular, consumers form a reference price that influences their demand. As in \cite{popescu2007dynamic}, we consider a profit-maximizing monopolist who faces a homogeneous stream of repeated customers over an infinite time horizon. In each period, the monopolist decides on a price $a \in A : =[0 ,\overline{a}]$ to charge the consumers. Assume for simplicity that the marginal cost is $0$. The resulting single-period payoff function is given by
\begin{equation*}r(s ,a) =aD(s ,a)
\end{equation*} where $s \in S \subseteq \mathbb{R}$ is the current reference price and $D(s ,a)$ is the demand function that depends on the reference price $s$ and on the price that the monopoly charges $a$. We assume that the function $D(s ,a)$ is continuous, non-negative, decreasing in $p$, increasing in $s$, has increasing differences, and is convex in $s$. If the current reference price is $s$ and the firm sets a price of $a$ then the next period's reference price is given by $\gamma s +(1 -\gamma )a$ (see \cite{popescu2007dynamic} for details on the micro foundations of this structure). $\gamma $ is called the memory factor. In contrast to the model of \cite{popescu2007dynamic}, we assume that the memory factor $\gamma $ is not deterministic. More precisely, we assume that the memory factor $\gamma $ is a random variable on $[0 ,1]$ with law $v$. So the transition probability function $p$ is given by \begin{equation*}p(s ,a ,B) =v\{\gamma  \in [0 ,1] :(\gamma s +(1 -\gamma )a) \in B\}
\end{equation*}for all $B \in \mathcal{B}(S)$. We show that even when the memory factor $\gamma $ is a random variable,  the result of \cite{popescu2007dynamic} holds in expectation, i.e., the long run expected prices are increasing in the current reference price. We also show that an increase in the discount factor increases the current optimal price and the long run expected prices.

\begin{proposition}\label{POPESCU}
Suppose that the function $D(s ,a)$ is continuous, non-negative, decreasing in $p$, increasing and convex in $s$, and has increasing differences.

(i) The optimal pricing policy $g(s)$ is increasing in the reference price $s$. 

(ii) The expected optimal prices in each period are higher when the initial reference price is higher.  

(iii)  $0 <\beta _{1} \leq \beta _{2} <1$ implies that $g(s ,\beta _{2}) \geq g(s ,\beta _{1})$ for all $s \in S$ and $\mathbb{E}_{2}^{t}(g(\beta _{2})) \geq \mathbb{E}_{1}^{t}(g(\beta _{1}))$ for all $t \in \mathbb{N}$.
\end{proposition}

\subsection{Controlled random walks} \label{Sec: cont ran walk}
Controlled random walks are used to study controlled queueing systems and other phenomena in applied probability (for example, see \cite{serfozo1981optimal}). In this section we consider a simple controlled random walk on $\mathbb{R}$. At any period, the state of the system $s \in \mathbb{R}$ determines the current period's reward $c_{1}(s)$. The next period's state is given by $m(s ,a ,\epsilon ) =a +s +\epsilon $ where $\epsilon $ is a random variable with law $v$ and support $\mathcal{V} \subseteq \mathbb{R}$, and $a \in A$ is the action that the DM chooses. Thus, the process evolves as a random walk $s +\epsilon$ plus the DM's action $a$. When the DM chooses an action $a \in A$, a cost of $c_{2}(a)$ is incurred. We assume that $A \subseteq \mathbb{R}$ is a compact  set, $c_{1}(s)$ is an increasing and convex function, and $c_{2}$ is an increasing function. That is, the reward and the marginal reward are increasing in the state of the system and the costs are increasing in the action that the DM chooses.  

The single-period payoff function is given by $r(s ,a) =c_{1}(s) -c_{2}(a)$ and the transition probability function is given by \begin{equation*}p(s ,a ,B) =v\{\epsilon  \in \mathcal{V} :a +s +\epsilon  \in B\}
\end{equation*}for all $B \in \mathcal{B}(\mathbb{R})$. In this setting, when choosing an action $a$, the DM faces the following trade-off between the current payoff and future payoffs: while choosing a higher action $a$ has higher current costs, it increases the probability that the state of the system will be higher in the next period, and thus, a higher action increases the probability of higher future rewards.

We study how a change in the random variable $\epsilon$ affects the DM's current and future optimal decisions. When $c_{1}(s)$ is convex and increasing in $s$, it is easy to see that the transition function $m(s ,a ,\epsilon ) =a +s +\epsilon $ and the single-period function $r(s ,a) =c_{1}(s) -c_{2}(a)$ satisfy the conditions of Theorem \ref{Theorem Transition}. Thus, the proof of the following proposition follows immediately from Theorem \ref{Theorem Transition}.

\begin{proposition} \label{prop: Dist}
Suppose that $p_{i}(s ,a ,B) =\Pr (a +s +\epsilon _{i} \in B)$ where $\epsilon _{i}$ has the law $v_{i}$, $i =1 ,2$. Suppose that $c_{1}(s)$ is convex and increasing in $s$. Assume that   $v_{2} \succeq _{st}v_{1}$. 

Then $g(s ,p_{2}) \geq g(s ,p_{1})$ for all $s \in S$, $g(s ,p_{2})$ is increasing in $s$, and $\mathbb{E}_{2}^{t}(g(p_{2})) \geq \mathbb{E}_{1}^{t}(g(p_{1}))$ for all $t \in \mathbb{N}$.

\end{proposition}

\subsection{\label{Section 3.3}Comparisons of stationary distributions}
Stationary equilibrium is the preferred solution concept for many models that describe large dynamic economies (see \cite{acemoglu2015robust} for examples of such models). In these models, there is a continuum of agents. Each agent has an individual state and solves a discounted dynamic programming problem given some parameters $e$ (usually prices). The parameters are determined by the aggregate decisions of all agents. Informally, a stationary equilibrium of these models consists of a set of parameters $e$, a policy function $g$, and a probability measure $\lambda$ on $S$ such that (i) $g$ is an optimal stationary policy given the parameters $e$, (ii) $\lambda$ is a stationary distribution of the states' dynamics $P(s ,B)$ given the parameters $e$, and (iii) the parameters $e$ are determined as a function of $\lambda $ and $g$.\protect\footnote{ Stationary equilibrium models are used to study a wide range of economic phenomena. Examples include models of industry equilibrium \citep{hopenhayn1992entry}, heterogeneous agent macro models   \citep{huggett1993risk} and \citep{aiyagari1994uninsured}, and many more.} 

The existence and uniqueness of a stationary probability measure $\lambda $ on $S$ in the sense that \begin{equation*}\lambda (B) =\int _{S}p(s ,g(s) ,B)\lambda (ds)
\end{equation*} for all $B \in \mathcal{B}(S)$ are widely studied.\protect\footnote{
For example, see \cite{hopenhayn1992stochastic},  \cite{kamihigashi2014stochastic}, and \cite{foss2018stochastic}.
} We now derive comparative statics results relating to how the stationary distribution $\lambda $ changes when the transition function $p$ changes. We denote the least stationary distribution by $\underline{\lambda}$ and the greatest stationary distribution by $\overline{\lambda}$.

\begin{proposition}
\label{Prop Comp Stationary}Suppose that $S$ is a compact set in $\mathbb{R}$.

(i) Let $E_{p ,i}$ be the set of all monotone transition probability functions $p$. Assume that $g(s ,p)$ is increasing  in $(s,p)$ on $S \times E_{p ,i}$ where $E_{p ,i}$ is endowed with the order $ \succeq _{st}$. Then the greatest stationary distribution $\overline{\lambda}$ and the least stationary distributions $\underline{\lambda}$ are increasing in $p$ on $E_{p ,i}$ with respect to $ \succeq _{st}$.\protect\footnote{
The existence of the greatest fixed point is guaranteed from the Tarski fixed-point theorem. For more details, see the Appendix and \cite{topkis2011supermodularity}.}   

(ii) Let $E_{p ,ic}$ be the set of all monotone and convexity-preserving transition probability functions $p$. Assume that $g(s,p)$ is convex in $s$ and is increasing in $(s,p)$ on $S \times E_{p ,ic}$ where $E_{p ,ic}$ is endowed with the order $ \succeq _{CX}$. Then the greatest stationary distribution $\overline{\lambda}$ and the least stationary distributions $\underline{\lambda}$ are increasing in $p$ on $E_{p ,ic}$ with respect to $ \succeq _{ICX}$.
\end{proposition}

We apply Proposition \ref{Prop Comp Stationary} to a standard stationary equilibrium model \citep{huggett1993risk}. 

There is a continuum of ex-ante identical agents with mass $1$. The agents solve a consumption-savings problem when their income is fluctuating. Each agent's payoff function is given by $r(s ,a) =u(s-a)$ where $s$ denotes the agent's current wealth, $a$ denotes the agent's savings, $s -a$ is the agent's current consumption, and $u$ is the agent's utility function. Thus, when an agent consumes $s-a$, his single-period payoff is given by $u(s-a)$.\footnote{For simplicity we assume that all the agents are ex-ante identical, i.e., the agents have the same utility function and transition function. The model can be extended to the case of ex-ante heterogeneity.} Recall that a utility function is in the class of hyperbolic absolute risk aversion (HARA) utility functions if its absolute risk aversion $A \left (c\right )$ is hyperbolic. That is, if $A (c):= -\frac{u^{ \prime  \prime } \left (c\right )}{u^{ \prime } \left (c\right )} =\frac{1}{ac+b}$ for $c >\frac{ -b}{a}$. We assume that $u$ is in the HARA class and that the utility function's derivative $u^{\prime}$ is convex.

Savings are limited to a single risk-free bond. When the agents save an amount $a$ their next period's wealth is given by $Ra +y$ where $R$ is the risk-free bond's rate of return and $y \in Y =[\underline{y} ,\overline{y}] \subset \mathbb{R}_{ +}$ is the agents' labor income in the next period. The agents' labor income is a random variable with law $\nu$. Thus, the transition function is given by \begin{equation*}p(s ,a ,B) =\ensuremath{\operatorname*{}}\nu \{y \in Y :Ra +y \in B \}.
\end{equation*}

The set from which the agents can choose their savings level is given by $\Gamma (s) =[\underline{s} ,\min \{s ,\overline{s}\}]$ where $\underline{s} <0$ is a borrowing limit and $\overline{s} >0$ is an upper bound on savings.

A stationary equilibrium is given by a probability measure $\lambda $ on $S =[\underline{s} ,(1 +r)\overline{s} +\overline{y}]$, a rate of return $R$, and a stationary savings policy function $g$ such that (i) $g$ is optimal given $R$, (ii) $\lambda $ is a stationary distribution given $R$, i.e., $\lambda (B) =\int _{S} p(s,g(s),B) \lambda (ds)$, and (iii) markets clear in the sense that the total supply of savings equals the total demand for savings, i.e., $\int g(s)\lambda (ds) =0$.

If the agents' utility function is in the HARA class then the savings policy function $g(s)$ is convex and increasing (see \cite{jensen2017distributional}). It is easy to see that $p$ is convexity-preserving and monotone. Furthermore, when  $u^{\prime }$ is convex then the policy function $g(s,p)$ is increasing in $p$ with respect to the convex order, i.e., $g(s ,p_{2}) \geq g(s ,p_{1})$ whenever $p_{2} \succeq _{CX}p_{1}$ (see \cite{light2018precautionary}). Thus, part (ii) of Proposition \ref{Prop Comp Stationary} implies that when the labor income uncertainty increases (i.e.,  $p_{2} \succeq _{CX}p_{1}$), both the highest partial equilibrium (when $R$ is fixed) wealth inequality and the lowest partial equilibrium wealth inequality increase (i.e., $\lambda _{2} \succeq _{ICX}\lambda _{1}$).

\section{\label{Section final}Summary}

This paper studies how the current and future optimal decisions change as a function of the optimization problem's parameters in the context of Markov decision processes. We provide simple sufficient conditions on the primitives of Markov decision processes that ensure comparative statics results and stochastic comparative statics results. 
We show that various models from different areas of operations research and economics satisfy our sufficient conditions.

\section{\label{Section Appendix} Appendix}

\subsection{Proofs of the results in Section \ref{SCS}}

\begin{proof}
[Proof of Theorem~\ref{Theorem 1}]
For $t =1$ the result is trivial since $\mu _{2}^{1} =\mu _{1}^{1}$. Assume that  $\mu _{2}^{t} \succeq _{D}\mu _{1}^{t}$ for some $t \in \mathbb{N}$. First note that for every measurable function $f :S \rightarrow \mathbb{R}$ and $i =1 ,2$ we have
\begin{equation}\int _{S}f(s^{ \prime })\mu _{i}^{t +1}(ds^{ \prime }) =\int _{S}\int _{S}f(s^{ \prime })P_{i}(s ,ds^{ \prime })\mu _{i}^{t}(ds) . \label{(1)}
\end{equation}To see this, assume first that $f =1_{B}$ where $B \in \mathcal{B}(S)$ and $1$ is the indicator function of the set $B$. We have 
\begin{align*}\int _{S}f(s^{ \prime })\mu _{i}^{t +1}(ds^{ \prime }) &  =\mu _{i}^{t +1}(B) \\
 &  =\int _{S}p_{i}(s ,g(s ,e_{i}) ,B)\mu _{i}^{t}(ds) \\
 &  =\int _{S}\int _{S}1_{B}(s^{ \prime })p_{i}(s ,g(s ,e_{i}) ,ds^{ \prime })\mu _{i}^{t}(ds) \\
 &  =\int _{S}\int _{S}f(s^{ \prime })P_{i}(s ,ds^{ \prime })\mu _{i}^{t}(ds) .\end{align*} A standard argument shows that equality (\ref{(1)}) holds for every measurable function $f$.

Now assume that $f \in D$. We have 
\begin{align*}\int _{S}f(s^{ \prime })\mu _{2}^{t +1}(ds^{ \prime }) 
&  =\int _{S}\int _{S}f(s^{ \prime })P_{2}(s ,ds^{ \prime })\mu _{2}^{t}(ds)  \\
&  \geq \int _{S}\int _{S}f(s^{ \prime })P_{2}(s ,ds^{ \prime })\mu _{1}^{t}(ds) \\
 &  \geq \int _{S}\int _{S}f(s^{ \prime })P_{1}(s ,ds^{ \prime })\mu _{1}^{t}(ds) \\
 & = \int _{S}f(s^{ \prime })\mu _{1}^{t +1}(ds^{ \prime }) .\end{align*}The first inequality follows since $f \in D$, $P_{2}$ is $D$-preserving and $\mu _{2}^{t} \succeq _{D}\mu _{1}^{t}$ . The second inequality follows since $P_{2}(s , \cdot ) \succeq _{D}P_{1}(s , \cdot )$. Thus, $\mu _{2}^{t +1} \succeq _{D}\mu _{1}^{t +1}$. We conclude that $\mu _{2}^{t} \succeq _{D}\mu _{1}^{t}$ for all $t \in \mathbb{N}$.
\end{proof}

\begin{proof}
[Proof of Corollary \ref{Parameter}]
We show that $P_{2}$ is $I$-preserving and that $P_{2}(s , \cdot ) \succeq _{st}P_{1}(s , \cdot )$ for all $s \in S$. Let $f :S \rightarrow \mathbb{R}$ be an increasing function and let $e_{2} \succeq e_{1}$. 

Since $p$ is monotone and $g(s ,e_{2})$ is increasing in $s$, if $s_{2} \geq s_{1}$ then \begin{equation*}\int _{S}f(s^{ \prime })p(s_{2} ,g(s_{2} ,e_{2}) ,ds^{ \prime }) \geq \int _{S}f(s^{ \prime })p(s_{1} ,g(s_{1} ,e_{2}) ,ds^{ \prime }) .
\end{equation*}Thus, $P_{2}$ is $I$-preserving. 

Let $s \in S$. Since $g(s ,e_{2}) \geq g(s ,e_{1})$ and $p$ is monotone, we have \begin{equation*}\int _{S}f(s^{ \prime })p(s ,g(s ,e_{2}) ,ds^{ \prime }) \geq \int _{S}f(s^{ \prime })p(s ,g(s ,e_{1}) ,ds^{ \prime }) .
\end{equation*}Thus, $P_{2}(s , \cdot ) \succeq _{st}P_{1}(s , \cdot )$.

From Theorem \ref{Theorem 1} we conclude that $\mu _{2}^{t} \succeq _{st}\mu _{1}^{t}$ for all $t \in \mathbb{N}$. We have \begin{equation*}\int _{S}g(s ,e_{2})\mu _{2}^{t}(ds) \geq \int _{S}g(s ,e_{2})\mu _{1}^{t}(ds) \geq \int_{S}g(s ,e_{1})\mu _{1}^{t}(ds) ,
\end{equation*} which proves the Corollary.
\end{proof}

\begin{proof}
[Proof of Theorem \ref{TRANSITION}]
(i) Assume that $p_{2} \succeq _{st}p_{1}$. We show that $P_{2}$ is $I$-preserving and that $P_{2}(s , \cdot ) \succeq _{st}P_{1}(s , \cdot )$ for all $s \in S$. Let $f :S \rightarrow \mathbb{R}$ be an increasing  function. 

Assume that $s_{2} \geq s_{1}$. Since $g(s_{2} ,p_{2}) \geq g(s_{1} ,p_{2})$ and $p_{2}$ is monotone we have \begin{equation*}\int _{S}f(s^{ \prime })p_{2}(s_{2} ,g(s_{2} ,p_{2}) ,ds^{ \prime }) \geq \int _{S}f(s^{ \prime })p_{2}(s_{1} ,g(s_{1} ,p_{2}) ,ds^{ \prime }),
\end{equation*} which proves that $P_{2}$ is $I$-preserving. 

Let $s \in S$. Since $p_{2}$ is monotone, $g(s ,p_{2}) \geq g(s ,p_{1})$ for all $s \in S$, and $p_{2} \succeq _{st}p_{1}$ we have 
\begin{align*}\int _{S}f(s^{ \prime })p_{2}(s ,g(s ,p_{2}) ,s ,ds^{ \prime }) &  \geq \int _{S}f(s^{ \prime })p_{2}(s ,g(s ,p_{1}) ,ds^{ \prime }) \\
 &  \geq \int _{S}f(s^{ \prime })p_{1}(s ,g(s ,p_{1}) ,ds^{ \prime }),\end{align*} which proves that $P_{2}(s , \cdot ) \succeq _{st}P_{1}(s , \cdot )$ for all $s \in S$.  

From Theorem \ref{Theorem 1} we conclude that $\mu _{2}^{t} \succeq _{st}\mu _{1}^{t}$ for all $t \in \mathbb{N}$. Since $g(s ,p_{2})$ is increasing, we have \begin{equation*}\int _{S}g(s ,p_{2})\mu _{2}^{t}(ds) \geq \int _{S}g(s ,p_{2})\mu _{1}^{t}(ds) \geq \int g(s ,p_{1})\mu _{1}^{t}(ds) ,
\end{equation*}which proves part (i).

(ii) Assume that $p_{2} \succeq _{CX}p_{1}$. We show that $P_{2}$ is $ICX$-preserving and that $P_{2}(s , \cdot ) \succeq _{ICX}P_{1}(s , \cdot )$ for all $s \in S$.

  Let $f :S \rightarrow \mathbb{R}$ be an increasing and convex function. Let $s_{1} ,s_{2} \in S$ and $s_{\lambda } =\lambda s_{1} +(1 -\lambda )s_{2}$ for $0 \leq \lambda  \leq 1$. We have
\begin{align*}\lambda \int _{S}f(s^{ \prime })p_{2}(s_{1},g(s_{1} ,p_{2}) ,ds^{ \prime }) &  +(1 -\lambda )\int _{S}f(s^{ \prime })p_{2}(s_{2} ,g(s_{2} ,p_{2}) ,ds^{ \prime }) \\
 &  \geq \int _{S}f(s^{ \prime })p_{2}(s_{\lambda } ,\lambda g(s_{1} ,p_{2}) +(1 -\lambda )g(s_{2} ,p_{2}) ,ds^{ \prime }) \\
 &  \geq \int _{S}f(s^{ \prime })p_{2}(s_{\lambda } ,g(s_{\lambda } ,p_{2}) ,ds^{ \prime }) .\end{align*}The first inequality follows since $p_{2}$ is convexity-preserving. The second inequality follows since $g(s ,p_{2})$ is convex and $p_{2}$ is monotone. Thus, $\int _{S}f(s^{ \prime })P_{2}(s ,ds^{ \prime })$ is convex. Part (i) shows that $\int _{S}f(s^{ \prime })P_{2}(s ,ds^{ \prime })$ is increasing. We conclude that $P_{2}$ is $ICX$-preserving.

Fix $s \in S$. We have 
\begin{align*}\int _{S}f(s^{ \prime })p_{2}(s ,g(s ,p_{2}) ,ds^{ \prime }) \geq \int _{S}f(s^{ \prime })p_{2}(s ,g(s ,p_{1}) ,ds^{ \prime }) \\
 \geq \int _{S}f(s^{ \prime })p_{1}(s ,g(s ,p_{1}) ,ds^{ \prime }) .\end{align*} The first inequality follows since $g(s ,p_{2}) \geq g(s ,p_{1})$ and $p_{2}$ is monotone. The second inequality follows since $p_{2} \succeq _{CX}p_{1}$. We conclude that $P_{2}(s , \cdot ) \succeq _{ICX}P_{1}(s , \cdot )$. 

From Theorem \ref{Theorem 1} we conclude that $\mu _{2}^{t} \succeq _{ICX}\mu _{1}^{t}$ for all $t \in \mathbb{N}$. Since $g(s ,p_{2})$ is increasing and convex, we have \begin{equation*}\int _{S}g(s ,p_{2})\mu _{2}^{t}(ds) \geq \int _{S}g(s ,p_{2})\mu _{1}^{t}(ds) \geq \int g(s ,p_{1})\mu _{1}^{t}(ds) ,
\end{equation*}which proves part (ii).    
\end{proof}

\subsection{Proofs of the results in Section \ref{Section: discount}}
In order to prove Theorem \ref{Thorem DISCOUNT} we need the following two results: 

\begin{proposition}
\label{TOPKIS} Suppose that Assumption \ref{Ass Topkis} holds. Then 

(i) $h(s ,a ,f)$ has increasing differences  whenever $f$ is an increasing function. 

(ii) $G(s)$ is ascending. In particular, $g(s) =\max G(s)$ is an increasing function. 

(iii) $Tf(s) =\max _{a \in \Gamma (s)}h(s ,a ,f)$ is an increasing function whenever $f$ is an increasing function.
$V(s)$ is an increasing function.  
\end{proposition}

\begin{proof} 
See Theorem 3.9.2 in \cite{topkis2011supermodularity}. 
\end{proof}

\begin{proposition}
\label{LOVEJOY}Let $(E , \succeq )$ be a partially ordered set. Assume that $\Gamma(s)$ is ascending. If $h(s ,a ,e ,f)$ has increasing differences in $(s ,a)$, $(s ,e)$, and $(a ,e)$, then \begin{equation*}Tf(s ,e) =\max _{a \in \Gamma (s)}h(s ,a ,e ,f)
\end{equation*} has increasing differences in $(s ,e)$.    
\end{proposition}

\begin{proof}
See Lemma 1 in \cite{hopenhayn1992stochastic} or Lemma 2 in \cite{lovejoy1987ordered}. 
\end{proof}

\begin{proof}
[Proof of Theorem \ref{Thorem DISCOUNT}]
(i) Let $E =(0 ,1)$ be the set of all possible discount factors, endowed with the standard order: $\beta _{2} \geq \beta _{1}$ if $\beta _{2}$ is greater than or equal to $\beta _{1}$. Assume that $\beta _{1} \leq \beta _{2}$. Let $f \in B(S \times E)$ and assume that $f$ has increasing differences in $(s ,\beta )$ and is increasing in $s$. Let $a_{2} \geq a_{1}$. Since $f$ has increasing differences, the function $f(s ,\beta _{2}) -f(s ,\beta _{1})$ is increasing in $s$. Since $p$ is monotone we have\begin{equation*}\int _{S}(f(s^{ \prime } ,\beta _{2}) -f(s^{ \prime } ,\beta _{1}))p(s ,a_{2} ,ds^{ \prime }) \geq \int _{S}(f(s^{ \prime } ,\beta _{2}) -f(s^{ \prime } ,\beta _{1}))p(s ,a_{1} ,ds^{ \prime }) .
\end{equation*}Rearranging the last inequality yields \begin{equation*}\int _{S}f(s^{ \prime } ,\beta _{2})p(s ,a_{2} ,ds^{ \prime }) -\int _{S}f(s^{ \prime } ,\beta _{2})p(s ,a_{1} ,ds^{ \prime }) \geq \int _{S}f(s^{ \prime } ,\beta _{1})p(s ,a_{2} ,ds^{ \prime }) -\int _{S}f(s^{ \prime } ,\beta _{1})p(s ,a_{1} ,ds^{ \prime }) .
\end{equation*}Since $f$ is increasing in $s$ and $p$ is monotone, the right-hand-side and the left-hand-side of the last inequality are nonnegative. Thus, multiplying the left-hand-side of the last inequality by $\beta _{2}$ and the right-hand-side of the last inequality by $\beta _{1}$ preserves the inequality. Adding to each side of the last inequality $r(a_{2} ,s) -r(a_{1} ,s)$ yields \begin{equation*}h(s ,a_{2} ,\beta _{2} ,f) -h(s ,a_{1} ,\beta _{2} ,f) \geq h(s ,a_{2} ,\beta _{1} ,f) -h(s ,a_{1} ,\beta _{1} ,f) .
\end{equation*}That is, $h$ has increasing differences in $(a ,\beta )$. An analogous argument shows that $h$ has increasing differences in $(s ,\beta )$. Proposition \ref{TOPKIS} guarantees that $h$ has increasing differences in $(s ,a)$ and that $Tf$ is increasing in $s$. 

Proposition \ref{LOVEJOY} implies that $Tf$ has increasing differences. We conclude that for all $n =1 ,2 ,3....$, $T^{n}f$ has increasing differences and is increasing in $s$. From standard dynamic programming arguments, $T^{n}f$ converges uniformly to $V$. Since the set of functions that has increasing differences and is increasing in $s$ is closed under uniform convergence, $V$ has increasing differences and is increasing in $s$. From the same argument as above, $h(s ,a ,\beta  ,V)$ has increasing differences in $(a ,\beta )$. Theorem 6.1 in \cite{topkis1978minimizing} implies that $g(s ,\beta )$ is increasing in $\beta $ for all $s \in S$. Proposition \ref{TOPKIS} implies that $g(s ,\beta )$ is increasing in $s$ for all $\beta  \in E$. We now apply Corollary \ref{Parameter} to conclude that $\mathbb{E}_{2}^{t}(g(\beta _{2})) \geq \mathbb{E}_{1}^{t}(g(\beta _{1}))$ for all $t \in \mathbb{N}$.

(ii) The proof is similar to the proof of part (i) and is therefore omitted.
\end{proof}

\subsection{Proofs of the results in Section \ref{Section: transition}}

\begin{proof}
[Proof of Theorem \ref{Theorem Transition}]
Suppose that the function $f \in B(S \times E_{p})$ is convex and increasing in $s$, and has increasing differences where $E_{p}$ is endowed with the stochastic dominance order $ \succeq _{st}$. Let $v_{2} \succeq _{st}v_{1}$.

Note that $m$ has increasing differences in $(s ,a)$, $(s ,\epsilon )$ and $(a ,\epsilon )$ if and only if $m$ is supermodular (see Theorem 3.2 in \cite{topkis1978minimizing}).  

From the fact that the composition of a convex and increasing function with a convex, increasing and supermodular function is convex and supermodular (see \cite{topkis2011supermodularity}) the function $f(m(s ,a ,\epsilon ) ,p_{2})$ is convex and supermodular in $(s ,a)$ for all $\epsilon  \in \mathcal{V}$. Since convexity and supermodularity are preserved under integration, the function $\int f(m(s ,a ,\epsilon ) ,p_{2})v_{2}(d\epsilon )$ is convex and supermodular  in $(s ,a)$. Thus, \begin{equation}h(s ,a ,p_{2} ,f) =r(s ,a) +\beta \int _{\mathcal{V}}f(m(s ,a ,\epsilon ) ,p_{2})v_{2}(d\epsilon )
\end{equation}is convex and supermodular in $(s ,a)$ as the sum of convex and supermodular functions. This implies that $Tf(s ,p_{2}) =\max _{a \in \Gamma (s)}h(s ,a ,p_{2} ,f)$ is convex. Since $h$ is increasing in $s$ it follows that $Tf(s ,p_{2})$ is increasing in $s$. 

 Note that for any increasing function $\overline{f} :S \rightarrow \mathbb{R}$ we have \begin{equation*}\int _{S}\overline{f}(s^{ \prime })p_{2}(s ,a ,ds^{ \prime }) =\int _{\mathcal{V}}\overline{f}(m(s ,a ,\epsilon ))v_{2}(d\epsilon ) \geq \int _{\mathcal{V}}\overline{f}(m(s ,a ,\epsilon ))v_{1}(d\epsilon ) =\int _{S}\overline{f}(s^{ \prime })p_{1}(s ,a ,ds^{ \prime }),
\end{equation*}
so $p_{2} \succeq _{st}p_{1}$. 

Fix $a \in A$, and let $s_{2} \geq s_{1}$. Since $f(m(s ,a ,\epsilon ) ,p_{2})$ is supermodular in $(s ,\epsilon )$, the function $f(m(s_{2} ,a ,\epsilon ) ,p_{2}) -f(m(s_{1} ,a ,\epsilon ) ,p_{2})$ is increasing in $\epsilon $. 
We have
\begin{align*}\int _{\mathcal{V}}(f(m(s_{2} ,a ,\epsilon ) ,p_{2}) -f(m(s_{1} ,a ,\epsilon ) ,p_{2}))v_{2}(d\epsilon ) &  \geq \int _{\mathcal{V}}(f(m(s_{2} ,a ,\epsilon ) ,p_{2}) -f(m(s_{1} ,a ,\epsilon ) ,p_{2}))v_{1}(d\epsilon ) \\
 &  \geq \int _{\mathcal{V}}(f(m(s_{2} ,a ,\epsilon ) ,p_{1}) -f(m(s_{1} ,a ,\epsilon ) ,p_{1}))v_{1}(d\epsilon ).\end{align*}The first inequality follows since $v_{2} \succeq _{st}v_{1}$. The second inequality follows from the facts that $m$ is increasing in $s$ and $f$ has increasing differences. Adding $r(s_{2} ,a) -r(s_{1} ,a)$ to each side of the last inequality implies that $h$ has increasing differences in $(s ,p)$. Similarly, we can show that $h$ has increasing differences in $(a ,p)$.   

Proposition \ref{LOVEJOY} implies that $Tf$ has increasing differences. We conclude that for all $n =1 ,2 ,3....$, $T^{n}f$ is convex and increasing in $s$ and has increasing differences. From standard dynamic programming arguments, $T^{n}f$ converges uniformly to $V$. Since the set of functions that have increasing differences and are convex and increasing in $s$ is closed under uniform convergence, $V$ has increasing differences and is convex and increasing in $s$. From the same argument as above, $h(s ,a ,p ,V)$ has increasing differences in $(a ,p)$ and $(s ,a)$. An application of Theorem 6.1 in \cite{topkis1978minimizing} implies that $g(s ,p_{2}) \geq g(s ,p_{1})$ for all $s \in S$ and $g(s,p_{2})$ is increasing in $s$. The fact that $m$ is increasing implies that $p$ is monotone. We now apply Corollary \ref{Parameter} to conclude that $\mathbb{E}_{2}^{t}(g(p_{2})) \geq \mathbb{E}_{1}^{t}(g(p_{1}))$ for all $t \in \mathbb{N}$.
\end{proof}

\subsection{Proofs of the results in Sections \ref{Section Popescu} and \ref{Section 3.3}}
\begin{proof}
[Proof of Proposition \ref{POPESCU}]
(i) Let $f \in B(S)$ be a convex function. The facts that $D(s ,a)$ is convex in $s$ and that convexity is preserved under integration imply that the function $aD(s ,a) +\beta \int f(\gamma s +(1 -\gamma )a)v(d\gamma )$ is convex in $s$. Thus, the function $Tf(s)$ given by 
\begin{equation*}Tf(s) =\max _{a \in A}aD(s ,a) +\beta \int f(\gamma s +(1 -\gamma )a)v(d\gamma )
\end{equation*} is convex in $s$. A standard dynamic programming argument (see the proof of Proposition \ref{prop: Dist}) shows that the value function $V$ is convex. The convexity of $V$ implies that for all $\gamma $, the function $V(\gamma s +(1 -\gamma )a)$ has increasing differences in $(s ,a)$. Since increasing differences are preserved under integration, $\int _{0}^{1}V(\gamma s +(1 -\gamma )a)v(d\gamma )$ has increasing differences in $(s ,a)$. Since $D(s ,a)$ is nonnegative and has increasing differences, the function $aD(s ,a)$ has increasing differences. Thus, the function \begin{equation*}aD(s ,a) +\beta \int _{0}^{1}V(\gamma s +(1 -\gamma )a)v(d\gamma )
\end{equation*}has increasing differences as the sum of functions with increasing differences. Now apply Theorem 6.1 in \cite{topkis1978minimizing} to conclude that $g(s)$ is increasing. 

(ii) Follows from Corollary \ref{Parameter}. 

(iii) Follows from a similar argument to the arguments in the proof of Theorem \ref{Thorem DISCOUNT}.  
\end{proof}

We now introduce some notations and a result that is needed in order to prove Proposition \ref{Prop Comp Stationary}. 
Recall that a partially ordered set $(Z , \geq )$ is said to be a lattice if for all $x ,y \in Z$, $\sup \{x ,y\}$ and $\inf \{y ,x\}$ exist in $Z$. $(Z , \geq )$ is a complete lattice if for all non-empty subsets $Z^{ \prime } \subseteq Z$ the elements $\sup Z^{ \prime }$ and $\inf Z^{ \prime }$ exist in $Z$. We need the following Proposition regarding the comparison of fixed points. For a proof, see Corollary 2.5.2 in \cite{topkis2011supermodularity}.

\begin{proposition}
\label{Topkis Fixed point}Suppose that $Z$ is a nonempty complete lattice, $E$ is a partially ordered set, and $f (z ,e)$ is an increasing function from $Z \times E$ into $Z$. Then the greatest and least fixed points 
of $f (z ,e)$ exist and are increasing in $e$ on $E$. 
\end{proposition}

\begin{proof}
[Proof of Proposition \ref{Prop Comp Stationary}]
Let $\mathcal{P}(S)$ be the set of all probability measures on $S$. The partially ordered set $(\mathcal{P} (S) , \succeq _{st})$ and the partially ordered set $(\mathcal{P} (S) , \succeq _{ICX})$ are complete lattices when $S \subseteq \mathbb{R}$ is compact (see \cite{muller2006stochastic}).  

(i) Define the operator $\Phi  :\mathcal{P}(S) \times E_{p ,i} \rightarrow \mathcal{P}(S)$ by \begin{equation*}\Phi (\lambda  ,p)( \cdot ) =\int _{S}p(s ,g(s ,p) , \cdot )\lambda (ds) .
\end{equation*}The proof of Theorem \ref{TRANSITION} implies that $\Phi $ is an increasing function on $\mathcal{P}(S) \times E_{p ,i}$ with respect to $ \succeq _{st}$. That is, for $p_{1} ,p_{2} \in E_{p ,i}$ and $\lambda _{1} ,\lambda _{2} \in \mathcal{P}(S)$ we have $\Phi (\lambda _{2} ,p_{2}) \succeq _{st}\Phi (\lambda _{1} ,p_{1})$ whenever $p_{2} \succeq _{st}p_{1}$ and $\lambda _{2} \succeq _{st}\lambda _{1}$. Proposition \ref{Topkis Fixed point} implies the result. 

(ii) The proof is analogous to the proof of part (i) and is therefore omitted. 
\end{proof}

\bibliographystyle{ecta}
\bibliography{Comparativedynamics}

\end{document}